\documentclass[12pt]{article}
\usepackage{amssymb}
\usepackage{amsmath,amsthm}
\usepackage[latin1]{inputenc}
\usepackage{hyperref}
\usepackage{color}
\usepackage{graphicx}
\hypersetup{colorlinks=true, linkcolor=blue, citecolor=blue,
urlcolor=blue}

 \newtheorem{remark}{Remark}

 \newtheorem{lemma}[remark]{Lemma}
 \newtheorem{theorem}[remark]{Theorem}
 \newtheorem{proposition}[remark]{Proposition}
 \newtheorem{corollary}[remark]{Corollary}

\title{The partition dimension of corona product graphs}

\author{J. A.
Rodr\'{\i}guez-Vel\'{a}zquez$^{1}$, I. G.
Yero$^{1}$ and D. Kuziak$^{2}$\\
    \\
$^1${\small Departament d'Enginyeria Inform\`atica i Matem\`atiques,}\\
{\small Universitat Rovira i Virgili,}  {\small Av. Pa\"{\i}sos
Catalans 26, 43007 Tarragona, Spain.} \\{\small
juanalberto.rodriguez\@@urv.cat, ismael.gonzalez\@@urv.cat}
\\
$^2${\small Faculty of Applied Physics and  Mathematics}\\
{\small Gda\'nsk University of Technology,} {\small ul. Narutowicza
11/12 80-233 Gda\'nsk, Poland} \\ {\small
dkuziak\@@mif.pg.gda.pl}\\
}

\date{September 20, 2010}

\begin{document}

\maketitle

\begin{abstract}
Given a set of vertices $S=\{v_1,v_2,...,v_k\}$ of a connected graph
$G$, the metric  representation of a vertex $v$ of $G$ with respect
to $S$ is the vector $r(v|S)=(d(v,v_1),d(v,v_2),...,d(v,v_k))$,
where $d(v,v_i)$, $i\in \{1,...,k\}$ denotes the distance between
$v$ and $v_i$. $S$ is a resolving set of $G$ if for every pair of
vertices $u,v$ of $G$, $r(u|S)\ne r(v|S)$. The metric dimension
$dim(G)$ of $G$ is the minimum cardinality of any resolving set of
$G$. Given an ordered partition $\Pi =\{P_1,P_2, ...,P_t\}$ of
vertices of a connected graph $G$, the partition representation of a
vertex $v$ of $G$, with respect to the partition $\Pi$ is the vector
$r(v|\Pi)=(d(v,P_1),d(v,P_2),...,d(v,P_t))$, where $d(v,P_i)$,
$1\leq i\leq t$, represents the distance between the vertex $v$ and
the set $P_i$, that is $d(v,P_i)=\min_{u\in P_i}\{d(v,u)\}$. $\Pi$
is a resolving partition for $G$ if for every pair of vertices $u,v$
of $G$, $r(u|\Pi)\ne r(v|\Pi)$. The partition dimension $pd(G)$ of
$G$ is the minimum number of sets in any resolving partition for $G$.
Let $G$ and $H$ be two graphs of order $n_1$ and $n_2$ respectively.
The corona product $G\odot H$ is defined as the graph obtained from
$G$ and $H$ by taking one copy of $G$ and $n_1$ copies of $H$ and
then joining by an edge, all the vertices from the $i^{th}$-copy of
$H$ with the $i^{th}$-vertex of $G$. Here we study the relationship between $pd(G\odot H)$ and several parameters of the graphs $G\odot H$,  $G$ and $H$, including
$dim(G\odot H)$, $pd(G)$ and $pd(H)$.
\end{abstract}



\section{Introduction}
 The concepts of resolvability and location in graphs were described
independently by Harary and Melter \cite{harary} and Slater
\cite{leaves-trees}, to define the same structure in a
graph. After these papers were published several authors
developed diverse theoretical works about this topic
\cite{pelayo1,pelayo2,chappell,chartrand,chartrand1,chartrand2,fehr,landmarks,survey,tomescu}.
 Slater described the usefulness of these ideas into long range
aids to navigation \cite{leaves-trees}. Also, these concepts  have
some applications in chemistry for representing chemical compounds
\cite{pharmacy1,pharmacy2} or to problems of pattern recognition and
image processing, some of which involve the use of hierarchical data
structures \cite{Tomescu1}. Other applications of this concept to
navigation of robots in networks and other areas appear in
\cite{chartrand,robots,landmarks}. Some variations on resolvability
or location have been appearing in the literature, like those about
conditional resolvability \cite{survey}, locating domination
\cite{haynes}, resolving domination \cite{brigham} and resolving
partitions \cite{chappell,chartrand2,fehr}. In this work we are
interested into study the relationship between $pd(G\odot H)$ and several parameters of the graphs $G\odot H$,  $G$ and $H$, including
$dim(G\odot H)$, $pd(G)$ and $pd(H)$.

We begin by giving some basic concepts and notations. Let $G=(V,E)$
be a simple graph.  Let $u,v\in V$ be two different
vertices in $G$, the distance $d_G(u,v)$ between two vertices $u$
and $v$ of $G$ is the length of a shortest path between $u$ and $v$.
If there is no ambiguity, we will use the notation $d(u,v)$ instead
of $d_G(u,v)$.  The diameter of $G$ is defined as $D(G)=\max_{u,v\in
V}\{d(u,v)\}$. Given $u,v\in V$,  $u\sim v$ means that $u$ and $v$
are adjacent vertices. Given a set of vertices
$S=\{v_1,v_2,...,v_k\}$ of a connected graph $G$, the {\it metric
representation} of a vertex $v\in V$ with respect to $S$ is the
vector $r(v|S)=(d(v,v_1),d(v,v_2),...,d(v,v_k))$. We say that $S$ is
a {\it resolving set} for $G$ if for every pair of distinct vertices
$u,v\in V$, $r(u|S)\ne r(v|S)$. The {\it metric dimension} of $G$ is
the minimum cardinality of any resolving set for $G$, and it is
denoted by $dim(G)$.

Given an ordered partition $\Pi =\{P_1,P_2, ...,P_t\}$ of vertices
of a connected graph $G$, the {\it partition  representation} of a
vertex $v\in V$ with respect to the partition $\Pi$ is the vector
$r(v|\Pi)=(d(v,P_1),d(v,P_2),...,d(v,P_t))$, where $d(v,P_i)$,
$1\leq i\leq t$, represents the distance between the vertex $v$ and
the set $P_i$, that is $d(v,P_i)=\min_{u\in P_i}\{d(v,u)\}$. We say
that $\Pi$ is a {\it resolving partition} of $G$ if for every pair
of distinct vertices $u,v\in V$, $r(u|\Pi)\ne r(v|\Pi)$. The {\it
partition dimension} of $G$ is the minimum number of sets in any
resolving partition for $G$ and it is denoted by $pd(G)$. The
partition dimension of graphs is studied in
\cite{chappell,chartrand2,survey,tomescu,yero-jarv-pd-cartesian}.

Let $G$ and $H$ be two graphs of order $n_1$ and $n_2$,
respectively. The corona product $G\odot H$ is defined as the graph
obtained from $G$ and $H$ by taking one copy of $G$ and $n_1$ copies
of $H$ and joining by an edge each vertex from the $i^{th}$-copy of
$H$ with the $i^{th}$-vertex of $G$. We will denote by
$V=\{v_1,v_2,...,v_n\}$   the set of vertices of $G$ and by
$H_i=(V_i,E_i)$   the copy of $H$ such that $v_i\sim v$ for every
$v\in V_i$.

\section{Majorizing  $pd(G\odot H)$}
It was shown in \cite{chartrand2}  that for any nontrivial connected graph $G$ we have
$pd(G)\le  dim(G) + 1.$ Thus,
 \begin{equation}\label{PdDimCorona}
 pd(G\odot H)\le  dim(G\odot H) + 1.
 \end{equation}

 In order to give another interesting relationship between $pd(G\odot H)$ and $dim(G\odot H)$
 that allow us to derive tight bounds on $pd(G\odot H)$, we present the following lemma.

\begin{lemma}{\rm \cite{metricDimCorona}} \label{first-lemma-part-Corona}
Let $G=(V,E)$ be a connected graph of order $n\ge 2$ and let $H$ be
a graph of order at least two.  Let $H_i=(V_i,E_i)$ be the subgraph
of $G\odot H$ corresponding to the $i^{th}$ copy of $H$.

\begin{itemize}

\item[{\rm (i)}]  If $u,v\in V_i$, then $d_{G\odot H}(u,x)=d_{G\odot H}(v,x)$ for every vertex $x$ of $G\odot H$ not belonging to $V_i$.

\item[{\rm (ii)}] If  $S$ is a resolving set for $G\odot H$, then  $V_i\cap S\neq \emptyset$ for every $i\in \{1,...,n\}$.

\item[{\rm (iii)}]   If $S$ is a resolving set for $G\odot H$ of minimum cardinality, then  $V\cap S=\emptyset$.
\end{itemize}
\end{lemma}



\begin{theorem}\label{general-bound-pd-n1-n2}
Let $G$ be a connected graph of order $n_1\ge 2$ and let $H$ be a graph of order $n_2$. Then
$$pd(G\odot H)\le \frac{1}{n_1}dim(G\odot H)+pd(G)+1.$$
\end{theorem}

\begin{proof}
Let $S$ be a resolving set for $G\odot H$ of minimum cardinality. By
Lemma  \ref{first-lemma-part-Corona}  (ii) and (iii) we conclude
that $S=\cup_{i=1}^{n_1}S_i$, where $\emptyset\ne S_i\subset V_i$.
We note that $|S_i|=\frac{|S|}{n_1}=\frac{1}{n_1}dim(G\odot H)$ for
every $i\in \{1,...,n_1\}$. In order to build a resolving partition for $G\odot H$, we need to introduce some additional notation.  Let
$\Pi(G)=\{W_1,W_2,...,W_{pd(G)}\} $ be a resolving partition for $G$,
let $A=\cup_{i=1}^{n_1}(V_i-S_i)$, let $S_i=\{v_{i1}, v_{i2}, ...,
v_{it}\}$, and let $B_j=\cup_{i=1}^{n_1}\{v_{ij}\}$, $j=1,...,t$.
Let us prove that $\Pi=\{A,B_1,...,B_t,W_1,...,W_{pd(G)}\}$ is a
resolving partition for $G\odot H$. Let $x,y$ be two  different
vertices  of $G\odot H$. We have the following cases.

Case 1. $x,y\in V_i$.  If $x\in S_i$ or $y\in S_i$  then $x$ and $y$
belong to different sets of $\Pi$,  so $r(x| \Pi)\ne r(y|\Pi)$.   We
suppose $x,y\in V_i-S_i$.  Since $S$ is a resolving set for $G\odot
H$, we have $r(x|S)\ne r(y|S)$. By Lemma
\ref{first-lemma-part-Corona} (i), $d_{G\odot H}(x,u)=d_{G\odot
H}(y,u)$ for every vertex $u$ of $G\odot H$ not belonging to $V_i$.
So, there exists $v\in S_i$ such that $d_{G\odot H}(x,v)\ne
d_{G\odot H}(y,v)$.  Thus, either ($v\sim x$ and $v\not\sim y$) or
($v\not\sim x$ and $v \sim y$). In the first case we have $d_{G\odot
H}(x,v)=d_{H_i}(x,v)=1$ and $d_{G\odot H}(y,v)=2\le d_{H_i}(y,v)$.
The case $v\not\sim x$ and $v\sim y$ is analogous. Therefore, for
every $x,y\in V_i$ there exists $v_{il}\in S_i$ such that $d_{G\odot
H}(x,B_l)=d_{G\odot H}(x,v_{il})\ne d_{G\odot H}(y,v_{il})=d_{G\odot
H}(y,B_l)$.

Case 2. $x\in V_i$ and $y\in V_j$, $j\ne i$. There exists $W_k\in
\Pi(G)$ such that $d_G(v_i,W_k)\ne d_G(v_j,W_k)$.  Thus, $d_{G\odot
H}(x,W_k)=1+d_G(v_i,W_k)\ne d_G(v_j,W_k)+1=d_{G\odot H}(y,W_k)$.

Case 3. $x,y\in V$. There exists $W_k\in \Pi(G)$ such that
$d_G(x,W_k)\ne d_G(y,W_k)$. Thus, $d_{G\odot H}(x,W_k)\ne d_{G\odot
H}(y,W_k)$.

Case 4. $x\in V$ and $y\not\in V$. In this case $x$ and $y$ belong to different sets of $\Pi$, so $r(x| \Pi)\ne r(y|\Pi)$.

Therefore, $\Pi$ is a resolving partition for  $G\odot H$.
\end{proof}

We denote by $K_n$ and $P_n$ the complete graph and the path graph of order $n$, respectively.
The following proposition allows us to conclude that for every connected graphs $G$ and $H$  of order greater than or equal to two such that  $G\odot H\not\cong K_{n_1}\odot P_{2}$ and $G\odot H\not\cong K_{n_1}\odot P_{3}$,   the equation in Theorem \ref{general-bound-pd-n1-n2} is never worse than equation (\ref{PdDimCorona}).

\begin{proposition}
Let $G$ and $H$ be two  connected graph of order greater than or equal to two. Let $n_1$ denote the order of $G$. If $G\odot H\not\cong K_{n_1}\odot P_{2}$ and $G\odot H\not\cong K_{n_1}\odot P_{3}$, then
$$dim(G\odot H) \ge \frac{n_1}{n_1-1}pd(G).$$
\end{proposition}

\begin{proof}
It was shown in \cite{metricDimCorona} that
\begin{equation}\label{cotainfdim}
dim(G\odot H)\ge n_1dim(H).
\end{equation}
 So we differentiate two cases.
Case 1: $dim(H)\ge2$. Since $n_1\ge 2$, we have
$2n_1(n_1-1)\ge n_1^2$. Thus,
$$dim(H)n_1(n_1-1)\ge 2n_1(n_1-1)\ge n_1^2\ge n_1pd(G).$$  Hence, by equation (\ref{cotainfdim}) we obtain
$dim(G\odot H)(n_1-1)\ge n_1pd(G).$

Case 2: $dim(H)=1$. It was shown in \cite{chartrand} that a connected graph $H$  has dimension $1$ if and only if $H$ is a path graph. So we have $H\cong P_{n_2}$.
Now we consider two subcases.

Subcase 2.1: $G\not\cong K_{n_1}$ and $n_2\ge 2$. Then by equation (\ref{cotainfdim}) we obtain $$(n_1-1)dim(G\odot P_{n_2})\ge n_1(n_1-1)\ge n_1pd(G)$$ and, as a consequence, $dim(G\odot H) \ge \frac{n_1}{n_1-1}pd(G).$

Subcase 2.2: $G\cong K_{n_1}$ and $n_2\ge 4$. Let $S$ be  a resolving set for $K_{n_1}\odot P_{n_2}$ of minimum cardinality. As above we denote by $\{v_1,...,v_{n_1}\}$ the set of vertices of $K_{n_1}$ and by $H_i=(V_i,E_i)$, $i\in \{1,...,n_1\}$ the corresponding copies of $P_{n_2}$ in $K_{n_1}\odot P_{n_2}$. By Lemma \ref{first-lemma-part-Corona} (ii)  we know that
 $V_i\cap S\ne \emptyset$, for every $i\in \{1,...,n_1\}$. We suppose $V_i\cap S =\{x_i\}$. In this case, since $n_2\ge4$ and $H_i\cong P_{n_2}$, there exist
$a,b\in V_i$ such that either $d_{K_{n_1}\odot P_{n_2}}(a,x_i)=d_{K_{n_1}\odot P_{n_2}}(b,x_i)=1$ or $d_{K_{n_1}\odot P_{n_2}}(a,x_i)=d_{K_{n_1}\odot P_{n_2}}(b,x_i)=2$. Thus, By Lemma \ref{first-lemma-part-Corona} (i) we conclude that $r(a|S)=r(b|S)$, a contradiction. Hence, $|V_i\cap S| \ge 2$ and, as a consequence, $dim(K_{n_1}\odot P_{n_2})\ge 2n_1$. Then
$$dim(K_{n_1}\odot P_{n_2})(n_1-1)\ge 2n_1(n_1-1)\ge n_1^2= n_1pd(K_{n_1}).$$
Therefore, the result follows.
\end{proof}

In
\cite{metricDimCorona}  we showed that for every connected
graph $G$  of order $n_1\ge 2$ and every graph $H$ of order $n_2\ge
2$,
$$dim(G\odot H)\le \left\{\begin{array}{ll}
n_1(n_2-\alpha-1) & \textrm{for  $\alpha\ge 1$ and $\beta\ge 1$,}\\
\\
n_1(n_2-\alpha) & \textrm{for  $\alpha\ge 1$ and $\beta = 0$,}\\
\\
n_1(n_2-1) & \textrm{for  $\alpha =0$,}
\end{array}\right.
$$
where  $\alpha$ denotes the number of connected components of $H$  and  $\beta$ denotes the number of isolated vertices of $H$.

By using the above bound on $dim(G\odot H)$   we obtain the following direct consequence of Theorem \ref{general-bound-pd-n1-n2}.

\begin{corollary} \label{boundOrderPd}
Let $G$ be a connected graph of order $n_1\ge 2$ and let $H$ be a
graph of order $n_2\ge 2$.  Let  $\alpha$ be the number of connected
components of $H$ of order greater than one and let $\beta$ be the
number of isolated vertices of $H$. Then
$$pd(G\odot H)\le \left\{\begin{array}{ll}
pd(G)+n_2-\alpha & \textrm{for  $\alpha\ge 1$ and $\beta\ge 1$,}\\
\\
pd(G)+n_2-\alpha+1 & \textrm{for  $\alpha\ge 1$ and $\beta = 0$,}\\
\\
pd(G)+n_2 & \textrm{for  $\alpha =0$.}
\end{array}\right.
$$
\end{corollary}

The reader is referred to \cite{metricDimCorona} for several upper bounds on $dim(G\odot H)$ which lead to bounds on $pd(G\odot H)$.

\begin{theorem}\label{partition-plus-partition}
Let $G$ and $H$ be two connected graphs of order $n_1\ge 2$ and  $n_2\ge 2$, respectively.  If $D(H)\le 2$, then
$$pd(G\odot H)\le pd(G)+pd(H).$$
\end{theorem}

\begin{proof}
Let $P=\{A_1,A_2,...A_k\}$ be a resolving partition in $G$ and let
$Q_i=\{B_{i1},B_{i2},...B_{it}\}$ be a resolving  partition in the
corresponding copy $H_i$ of $H$. Let
$B_j=\bigcup_{i=1}^{n_1}B_{ij}$, $j\in \{1,...,t\}.$  We will show
that $$\Pi=\{A_1,A_2,...,A_k,B_1,B_2,...,B_t\}$$ is a resolving partition for $G\odot H$. Let $x,y$ be two different vertices of
$G\odot H$. If $x,y\in A_i$, then there exists $A_j\in P\subset
\Pi$, $j\ne i$, such that $d(x, A_j)\ne d(y,A_j)$. On the other
hand, if $x,y\in B_j$, then we have the following cases.

Case 1:  $x,y\in B_{ij}$. Hence, there exists $B_{ik}\in Q_i$, $k\ne
j$, such that $d_{H_i}(x,B_{ik})\ne d_{H_i}(y,B_{ik})$.  Since
$D(H)\le 2$, for every $u\in B_{ij}$ we have
$d_{H_i}(u,B_{ik})=d_{G\odot H}(u,B_{k})$ and
$d_{H_i}(u,B_{ik})=d_{G\odot H}(u,B_k)$. So, we obtain $d_{G\odot
H}(x,B_k)=d_{H_i}(x,B_{ik})\ne d_{H_i}(y,B_{ik})=d_{G\odot
H}(y,B_k).$

Case 2: $x\in B_{ij}$ and $y\in B_{kj}$, $k\ne i$.  If $v_i,v_k\in
A_l$, then there exists $A_q\in P\subset \Pi$ such  that
$d_G(v_i,A_q)\ne d_G(v_k,A_q)$. So, we have $d_{G\odot
H}(x,A_q)=1+d_G(v_i,A_q)\ne 1+d_G(v_k,A_q)=d_{G\odot H}(y,A_q).$

On the other hand, if $v_i\in A_p$ and $v_k\in A_q$, $q\ne p$, then we have
$d_{G\odot H}(x,A_q)=1+d_G(v_i,A_q)>1=d_G(y,A_q)=d_{G\odot H}(y,A_q).$

Thus, for every two different vertices $x,y$ of $G\odot H$ we have
$r(x|\Pi)\ne r(y|\Pi)$ and, as a  consequence, $\Pi$ is a resolving partition for $G\odot H$.
\end{proof}

\begin{corollary}\label{uperboundsumDim}
Let $G$ and $H$ be two connected graphs of order $n_1\ge 2$ and  $n_2\ge 2$, respectively. If $D(H)\le 2$, then
$$pd(G\odot H)\le dim(G)+dim(H)+2.$$
\end{corollary}

In the next section we will show that all the above inequalities are tight.

\section{Minorizing $pd(G\odot H)$}

\begin{theorem}\label{ThpdcoronaH}
Let $G$ and $H$ be two connected graphs. Let $\Pi$ be a resolving
partition  of $G\odot H$ of minimum cardinality. Let $H_i=(V_i,E_i)$
be the subgraph of $G\odot H$ corresponding to the $i^{th}$-copy of
$H$, and let $\Pi_i$ be the set composed by all non-empty sets of
the form $S\cap V_i$, where $S\in \Pi$. Then $\Pi_i$ is a resolving partition for $H_i$.
\end{theorem}

\begin{proof}
If $\Pi_i$ is composed by sets of cardinality one, then the result
immediately follows. Now, let $x,y$ be two  different vertices of
$H_i$ belonging to the same set of $\Pi$. We know that there exists
$S\in \Pi$ such that $d_{G\odot H}(x,S)\ne d_{G\odot H}(y,S)$.
  By Lemma \ref{first-lemma-part-Corona} (i) we have that for every vertex $v$ of
$G\odot H$ not belonging to $V_i$, it follows that $d_{G\odot
H}(x,v)=d_{G\odot H}(y,v)$. Hence we conclude $S'=S\cap V_i \ne
\emptyset$ and  we can assume, without loss of generality, that  $d_{G\odot
H}(x,S)=1$ and $d_{G\odot H}(y,S)=2$. As a result,  $S'\in \Pi_i$ and $d_{H_i}(x,S')=d_{G\odot H}(x,S)=1 <2 =d_{G\odot H}(y,S)\le d_{H_i}(y,S')$. Therefore, the result follows.
\end{proof}

\begin{corollary}
For any connected graphs $G$ and $H$, $$pd (G\odot H)\ge pd(H).$$
\end{corollary}

It is easy to check  that for the star graph $K_{1,n}$, $n\ge 2$, it follows $pd(K_{1,n})=n$.  So the following result shows that the above inequality is tight.

\begin{proposition}
Let $G$ denote a connected graph of order $n_1$ and let $n$ be an integer. If $n\ge 2n_1\ge 4$ or $n>2n_1=2$,  then $$pd (G\odot K_{1,n})= n.$$
\end{proposition}

\begin{proof}
Let us suppose $n\ge 2n_1\ge 4$.  For each $v_i\in V,$ let $\{a_i,u_{i1},u_{i2},...,u_{in}\}$  be the set
of vertices  of the $i^{th}$ copy of $K_{1,n}$ in $G\odot K_{1,n}$, where $a_i$ is the vertex of degree $n$.

We will show that $\Pi=\{S_1,
S_2,...,S_n\}$ is a resolving partition for $G\odot K_{1,n}$, where
$$\begin{array}{c}
S_1=\{a_1, u_{11},u_{21},...,u_{n_11}\}, \\
  S_2=\{v_1, u_{12},u_{22},...,u_{n_12}\}, \\
  S_3=\{a_2, u_{13},u_{23},...,u_{n_13}\}, \\
  S_4=\{v_2, u_{14},u_{24},...,u_{n_14}\}, \\
  \vdots \\
 S_{2n_1}=\{v_{n_1}, u_{1(2n_1)},u_{2(2n_1)},...,u_{n_1(2n_1)}\}, \\
  S_{2n_1+1}=\{u_{1(2n_1+1)},u_{2(2n_1+1)},...,u_{n_1(2n_1+1)}\}, \\
  \vdots \\
  S_{n}=\{u_{1n},u_{2n},...,u_{n_1n}\}.
\end{array}$$
Let $x,y$ be two different vertices of $G\odot K_{1,n}.$ We differentiate three cases.
Case 1: $x=u_{il}$ and $y=u_{jl}$, $i\ne j$. If $l\ne 2i-1$, then
$$d(u_{il},S_{2i-1})=d(u_{il},a_i)=1<2=d(u_{jl},u_{j(2i-1)})=d(u_{jl}, S_{2i-1}).$$
If $l= 2i-1$, then
$$d(u_{jl},S_{2j-1})=d(u_{jl},a_j)=1<2=d(u_{il},u_{i(2j-1)})=d(u_{il}, S_{2j-1}).$$
Case 2: $x=v_i$ and $y=u_{j(2i)}$. If $j=i$, then
$$d(v_i,S_{i})=d(v_i,u_{ii})=1<2=d(u_{i(2i)},u_{ii})=d(u_{i(2i)}, S_{i}).$$
If $j\ne i$, then
$$d(v_i,S_{i})=d(v_i,u_{ii})=1<2=d(u_{j(2i)},u_{ji})=d(u_{j(2i)}, S_{i}).$$
Case 3: $x=a_i$ and $y=u_{j(2i-1)}$. If $j=i$, then
$$d(a_i,S_{i})=d(a_i,u_{ii})=1<2=d(u_{i(2i-1)},u_{ii})=d(u_{i(2i-1)}, S_{i}).$$
If $j\ne i$, then
$$d(a_i,S_{i})=d(a_i,u_{ii})=1<2=d(u_{j(2i-1)},u_{ji})=d(u_{j(2i-1)}, S_{i}).$$
Therefore, we conclude that  $\Pi$ is a resolving partition for $G\odot K_{1,n}$.

For $n_1=1$   and $n\ge 3$ we denote by $v$ the vertex of $G$, by $a$ the
vertex of $K_{1,n}$ of degree $n$, and by $\{u_{1},u_{2},...,v_{n}\}$ the
set of leaves of $K_{1,n}$. Thus, from
$d(v,u_3)=1<2=d(u_2,u_3)$ and $d(a,u_3)=1<2=d(u_1,u_3)$, we conclude that
$\Pi=\{S_1, S_2,...,S_n\}$ is a resolving partition for $G\odot K_{1,n}$,
 where $S_1=\{a,u_1\}$, $S_2=\{v,u_2\}$,
$S_3=\{u_3\}$, ..., $S_n=\{u_n\}$.
\end{proof}

\begin{lemma}\label{LemmaDist3}
Let $G$ be a connected graph. If  $\Pi$ is a resolving partition for $G\odot K_n$ of cardinality $n+1$, then
for every vertex $v$ of $G\odot K_n$ and every $A\in \Pi$, it follows $d(v,A)\le 3$.
\end{lemma}

\begin{proof}
Let $v_i,v_j$ be two adjacent vertices of $G$ and let $H_l=(V_l,E_l)$ ($l\in \{i,j\}$)  be  the copy of $K_n$ in $G\odot K_n$ such that $v_l$ is adjacent to every
vertex of $H_l$.  If there exists a vertex $v$ of the subgraph of $G\odot K_n$ induced by $V_i\cup V_j\cup \{v_i,v_j\}$ such that $d(v,A)>3$, for some $A\in \Pi$, then,  since different  vertices of $V_i$ (respectively, $V_j$) belong to different sets of $\Pi$, there exist $B,C\in  \Pi$, $u_i\in V_i$ and $u_j\in V_j$ such that $u_i,v_i\in B$ and $u_j,v_j\in C$.

If $B=C$, then $d(u_i,A)=d(v_j,A)$ or $d(v_i,A)=d(u_j,A)$. Hence, $r(u_i|\Pi)=r(v_j|\Pi)$ or $r(v_i|\Pi)=r(u_j|\Pi)$, a contradiction. If $B\ne C$, then there exist two vertices  $u_i' \in V_i\cap C$ and  $u_j' \in V_j\cap B$ and, as a consequence, then $d(u_i',A)=d(v_j,A)$ or $d(v_i,A)=d(u_j',A)$. Thus,  $r(u_i'|\Pi)=r(v_j|\Pi)$ or $r(v_i|\Pi)=r(u_j'|\Pi)$, a contradiction. Therefore, $d(v,A)\le 3$, for every $A\in \Pi$.
\end{proof}

Given a graph $H$ which contains a connected component isomorphic to a complete
graph, we denote by $c(H)$ the  maximum cardinality of
any connected component of $H$  which is isomorphic to a complete
graph.

\begin{theorem}\label{teo-G-c(H)}
Let $G$ be a connected graph of order $n$. Then for any graph $H$ such that
$n > 2c(H)+1\ge 5$,
$$pd(G\odot H)\ge c(H)+2.$$
\end{theorem}

\begin{proof}
  We denote by $S_i$ a connected component  of $H_i$ isomorphic
to $K_{c(H)}$, $i\in \{1,...,n\}$. Since  different  vertices of
$S_i$  belong to different sets of any resolving partition for
$G\odot H$, we conclude $pd(G\odot H)\ge c(H)$. If $pd(G\odot
H)=c(H)$, then there exist two vertices $a,b\in S_i\cup \{v_i\}$
such that they belong to the same set of any resolving partition for
$G\odot H$. Thus, $a$ and $b$ have the same partition
representation, which is a contradiction. So, $pd(G\odot H)\ge
c(H)+1$. Now, let us suppose $pd(G\odot H)=c(H)+1$ and let
$\Pi(G\odot H)=\{A_1,A_2,...,A_{c(H)+1}\}$ be a resolving partition
for $G\odot H$. Now, let $S=\bigcup_{i=1}^{n}(S_i\cup \{v_i\})$ and
let  $u\in S$. Suppose $u\in A_j$, $j\in \{1,...,c(H)+1\}$. So, we
have that the partition representation of $u$ is given by
$$\begin{array}{ccccc}
    r(u|\Pi)=(1,1,..., & 1,0,1,& ..., &1,t,1, & ...,1), \\
     & j &   & i &
  \end{array}
$$
where $i,j\in \{1,...,c(H)+1\}$,  $i \ne j$,  and, by Lemma \ref{LemmaDist3}, $t\in \{1,2,3\}$.
Since  for every different vertices $a,b\in S$,
$r(a|\Pi)\ne r(b|\Pi)$, the maximum  number of possible
different partition representations for vertices of $S$ is given by
$(c(H)+1)(2c(H)+1),$ i.e., for  $t=1$ there are at most $c(H)+1$ different vectors and for $t\in \{2,3\}$ there are at most $2(c(H)+1)c(H)$.
Hence, $n(c(H)+1)=|S|\le (2c(H)+1)(c(H)+1)$ and, as a consequence, $n\le 2c(H)+1$. Therefore, if
$n > 2c(H)+1$, then $pd(G\odot H)\ge c(H)+2.$
\end{proof}

\begin{corollary} \label{lowerboundn+2}
Let $G$ be a  graph of order $n_1$ and let $n_2\ge 2$ be an integer. If
$n_1 > 2n_2+1$, then $$pd (G \odot K_{n_2})\ge n_2+2.$$
\end{corollary}

>From Theorem \ref{partition-plus-partition} and Corollary \ref{lowerboundn+2}
we obtain that if  $n_1 > 2n_2+1\ge 5$, then $pd(G)+n_2\ge pd (G \odot
K_{n_2})\ge n_2+2.$  Therefore, we obtain the following result.
\begin{remark}\label{Remarkn_2+2}
Let $n_1$ and  $n_2$ be  integers such that $n_1 > 2n_2+1\ge 5$.
 Then $$pd (P_{n_1} \odot K_{n_2})= n_2+2.$$
\end{remark}

By Remark \ref{Remarkn_2+2} we conclude that the inequalities in Theorem \ref{general-bound-pd-n1-n2}, Corollary \ref{boundOrderPd},  Theorem \ref{partition-plus-partition}, Corollary \ref{uperboundsumDim} and  Corollary \ref{lowerboundn+2} are tight.

An empty graph of order $n$, denoted by $N_n$,  consists of $n$ isolated nodes with no edges.
 In the following result $\beta(H)$ denotes the number of isolated vertices of a graph  $H$.

\begin{theorem}
Let $G$ be a connected graph of order $n\ge 2$ and let $H$ be any graph. If $n> \beta(H)\ge 2$, then $$pd(G\odot H)\ge \beta(H)+1.$$
\end{theorem}

\begin{proof}
We will proceed similarly to the proof of Theorem \ref{teo-G-c(H)}.
Let $S_i$ denote the set of isolated vertices of $H_i$, $i\in
\{1,...,n\}$.

 Since  different  vertices of $S_i$  belong to different sets of
any resolving partition for $G\odot H$, we have $pd(G\odot H)\ge
\beta(H)$. Let us suppose $pd(G\odot H)=\beta(H)$ and let
$\Pi(G\odot H)=\{A_1,A_2,...,A_{\beta(H)}\}$ be a resolving
partition for $G\odot H$. Now, let $S=\bigcup_{i=1}^{n}(S_i\cup
\{v_i\})$ and let $u\in S$. If $u\in A_j\cap S_j$, $j\in
\{1,...,n_1\}$, then the partition representation of $u$ is given by
$$\begin{array}{ccccc}
    r(u|\Pi)=(2,2,..., & 2,0,2,& ..., &2,t,2, & ...,2), \\
     & j &   & i &
  \end{array}
$$
with  $i,j\in \{1,...,\beta(H)\}$, $i\ne j$ and $t\in \{1,2\}$. On
the other side, if $u\in A_j\cap V$, then
$$\begin{array}{ccc}
    r(u|\Pi)=(1,1,..., & 1,0,1,& ...,1), \\
     & j &
  \end{array}
$$
with $j\in \{1,...,\beta(H)\}$. Thus, the maximum  number of
possible different partition representations for vertices of $S$ is
given by $(\beta(H)+1)\beta(H).$
Hence, $n(\beta(H)+1)=|S|\le \beta(H)(\beta(H)+1)$. Thus, $n\le
\beta(H)$. Therefore, if $n>\beta(H)$, then $pd(G\odot H)\ge
\beta(H)+1$.
\end{proof}

\begin{corollary} \label{lowerbound(n+1)Isolates}
Let $G$ be a  graph of order $n_1$ and let $n_2\ge 2$ be an integer. If
$n_1 > n_2$, then $$pd (G \odot N_{n_2})\ge n_2+1.$$
\end{corollary}

\begin{proposition}
If $n_1\ge n_2\ge 2$, then $$pd (P_{n_1}\odot N_{n_2})=n_2+1.$$
\end{proposition}

\begin{proof}
Let $V=\{v_1,...,v_n\}$ be the set of vertices of $P_{n_1}$ and, for each $v_i\in V$, let $V_i=\{u_{i1},...,u_{i{n_2}}\}$ be the set of vertices of the $i^{th}$ copy of $N_{n_2}$ in $P_{n_1}\odot N_{n_2}$. Let $\Pi=\{A_1,...,A_{n_2+1}\}$, where $A_1=\{v_1,u_{11}\}$, $A_2=\{v_i,u_{i1}: i\in \{2,...,n_1\}\}$ and $A_j=\{u_{i(j-1)}: i\in \{1,...,n_1\}\}$ for $j\in \{3,..,n_2+1\}$.
 Note that $d_{P_{n_1}\odot N_{n_2}}(v_1,A_2)\ne d_{P_{n_1}\odot N_{n_2}}(u_{11},A_2)$. Moreover, for two different vertices $x,y\in A_j$, $j\in \{3,...,n_{2}+1\}$,  we have  $d_{P_{n_1}\odot N_{n_2}}(x,A_1)\ne d_{P_{n_1}\odot N_{n_2}}(y,A_1)$. Now on we suppose $x,y\in A_2$. If $x,y\in V$ or $x,y\in V_i$, for some $i$, then $d_{P_{n_1}\odot N_{n_2}}(x,A_1)\ne d_{P_{n_1}\odot N_{n_2}}(y,A_1)$. Finally, if $x\in V$ and $y \not\in V$, then $d_{P_{n_1}\odot N_{n_2}}(x,A_3)\ne d_{P_{n_1}\odot N_{n_2}}(y,A_3)$. Therefore,
  $\Pi$ is a resolving partition for $P_{n_1}\odot N_{n_2}$ and, as a consequence, $pd (P_{n_1}\odot N_{n_2})\le n_2+1.$ By corollary \ref{lowerbound(n+1)Isolates} we conclude the proof.
\end{proof}

\section*{Acknowledgements}
  This work was partly supported  by the Spanish Ministry of Education through
projects TSI2007-65406-C03-01 \lq\lq E-AEGIS" and CONSOLIDER INGENIO
2010 CSD2007-00004 \lq\lq ARES''.

\end{document}